\documentclass[aap,MSNbibl,citesort,dvips]{arximspdf}
\usepackage[vtexbibl]{}

\doi{10.1214/10-AAP702}
\volume{21}
\issue{1}
\pubyear{2011}
\firstpage{374}
\lastpage{395}

\makeatletter

\def\calP{\mathcal{P}}

\def\calC{\mathcal{C}}

\def\Z{\mathbb{Z}}

\def\N{\mathbb{N}}

\def\al{\alpha}
\def\et{\eta}

\def\ep{\varepsilon}
\def\de{\delta}

\def\la{\lambda}
\def\La{\Lambda}

\def\si{\sigma}
\def\de{{\delta}}

\def\om{\omega}
\def\Om{\Omega}

\def\Ph1{P^{(h_1)}}
\def\Ph2{P^{(h_2)}}

\def\b0{{\mathbf 0}}

\def\bnu{\bar\nu}
\def\lbnu{{\bar\nu}_{\lambda}}
\def\1lbnu{{\bnu}_{\lambda_1}}
\def\2lbnu{{\bnu}_{\lambda_2}}

\def\lnbnu{{\lbnu}^{(n)}}
\def\qbnu{{\bnu}_{\langle q\rangle}}
\def\qnbnu{{\qbnu}^{(n)}}

\def\sixn{\sigma_x^{(n)}}
\def\etxn{\eta_x^{(n)}}

\def\hep{\hat{\varepsilon}}
\def\tep{\tilde{\varepsilon}}
\def\hN{\hat{N}}

\newtheorem{theorem}{Theorem}[section]
\newtheorem{prop}[theorem]{Proposition}
\newtheorem{lem}[theorem]{Lemma}
\newtheorem{cor}[theorem]{Corollary}
\newproclaim{clm}[theorem]{Claim}
\newproclaim{defn}[theorem]{Definition}
\newproclaim{rem}[theorem]{Remark}

\def\eqref#1{(\ref{#1})}
\makeatother

\begin{document}
\begin{frontmatter}

\title{Sharpness of the percolation transition in the two-dimensional contact process}
\runtitle{percolation transition in the 2D contact process}

\begin{aug}
\author[A]{\fnms{J.} \snm{van den Berg}\thanksref{t1}\ead[label=e1]{J.van.den.Berg@cwi.nl}\corref{}}
\runauthor{J. van den Berg}
\affiliation{CWI and VU University Amsterdam}
\thankstext{t1}{Supported in part by the Dutch BSIK/BRICKS project.}
\address[A]{CWI\\ Science Park 123\\ 1098 XG Amsterdam\\
The Netherlands\\
\printead{e1}} %adresu isvedimo komanda gale!
\end{aug}

% HISTORY:
\received{\smonth{7} \syear{2009}}
\revised{\smonth{3} \syear{2010}}

% ABSTRACT
\begin{abstract}
For ordinary (independent) percolation on a large class of lattices it is well known that below the
critical percolation parameter $p_c$ the cluster size distribution has exponential decay and that
power-law behavior of this distribution can only occur at $p_c$. This behavior is often called
``sharpness of the percolation transition.''

For theoretical reasons, as well as motivated by applied research, there is an increasing
interest in percolation models with (weak) dependencies. For instance, biologists and agricultural researchers
have used (stationary
distributions of) certain two-dimensional contact-like processes to model vegetation patterns in
an arid landscape (see \cite{Kef07}). In that context occupied clusters are interpreted as patches of vegetation.
For some of these models it is reported in \cite{Kef07} that computer simulations indicate
power-law behavior in some interval of positive length of a model parameter. This would mean that in
these models the percolation transition is not sharp.

This motivated us to investigate similar questions for the ordinary (``basic'') $2D$ contact process with
parameter  $\lambda$. We
show, using techniques from Bollob\'as  and Riordan \cite{BoRi06a,BoRi08}, that for
the upper invariant measure ${\bnu}_{\lambda}$ of this process the percolation transition
is sharp. If $\lambda$ is such that (${\bnu}_{\lambda}$-a.s.) there are
no infinite clusters,
then for all parameter values below $\lambda$ the cluster-size distribution has exponential decay.
\end{abstract}

% KEYWORDS
\begin{keyword}[class=AMS]
\kwd[Primary ]{60K35}
\kwd[; secondary ]{92D40}
\kwd{92D30}
\kwd{82B43}.
\end{keyword}
\begin{keyword}
\kwd{Percolation}
\kwd{contact process}
\kwd{sharp transition}
\kwd{approximate zero-one law}
\kwd{sharp thresholds}.
\end{keyword}

\end{frontmatter}

%s1 ###
\section{Introduction and statement of the main result}\label{sec1}

The contact process was introduced as a stochastic model for the spread of an infection in a population
with a geometric structure, usually represented by the $d$-dimensional cubic lattice.
Each vertex $x$ of this lattice represents an individual whose state, infected ($1$) or healthy ($0$), at time $t$ is denoted
by $\si_x(t)$. The dynamic in this model is as follows: A vertex in state $0$ goes to state $1$ (``becomes infected'') at a rate equal to
$\la$ times the number of neighbors of that vertex that are in state $1$. A vertex in state $1$ goes to state $0$ (``recovers'')
at rate $1$. Here $\la$ is the parameter of the model called the infection rate.
In this paper we restrict to the case $d=2$. Depending on the applications one has in mind the terms ``infected''
and ``healthy'' are sometimes replaced by ``occupied'' and ``vacant,'' respectively. In the remainder of this paper we will use this latter
terminology.

The configuration at time $t$ is denoted by $\si(t) := (\si_x(t), x \in \Z^2)$. Let $\mu_t$ denote
the distribution of $\si_t$ when we start at time $0$ with all vertices occupied.
We will use the notation $|V|$ for the cardinality of a set $V$.

It is well known (from a standard coupling argument) that $\mu_t$ is stochastically dominated by $\mu_s$ if
$s \leq t$.
Hence $\mu_t$ converges weakly to a probability measure denoted by $\bnu$ (=$\bnu_{\la}$) as $t \rightarrow \infty$.
This measure $\bnu$ is called the upper invariant measure.
It is well known (again by standard coupling arguments) that $\bnu_{\la_2}$ stochastically dominates
$\bnu_{\la_1}$ if $\la_2 > \la_1$.
Realizations are typically denoted by $ \si = (\si_x, x \in \Z^2)$.
The occupied cluster of a vertex $x$ (i.e., the maximal connected component which contains $x$ and of which every vertex is occupied) is
denoted by ${\calC}_x$. (If $x$ is the origin $0$, we often omit the subscript.)

In this paper we study the sizes of occupied clusters
under the measure $\bnu$.
Motivation comes from work by Liggett and Steif \cite{LiSt06} who showed that for $\la$ sufficiently large
percolation occurs [i.e.,
$\lbnu(|\calC| = \infty) > 0$]
and from work by biologists and agricultural researchers.
In this latter work (see \cite{Kef07}) limit distributions of contact-like processes
(more complicated than the ``basic process'' described above) were used
to model vegetation patterns
in arid regions in Spain and North Africa. In this ``agricultural'' context an occupied cluster is interpreted as a ``vegetation patch.''
For some of these models it was claimed in \cite{Kef07} that simulations suggest power-law behavior of the cluster size distribution in an
interval of some parameter.

In ordinary percolation models it is known that below the percolation threshold the distribution of the cluster size
has exponential decay and that power-law behavior can only occur at the percolation threshold. Triggered by the above-mentioned
claim in \cite{Kef07} concerning very different behavior in ``their''
contact-like processes, we study this question for $\lbnu$.
Before we state our main result, Theorem \ref{mainthm}, we give a brief and somewhat informal overview of earlier work on
exponential-decay results in percolation to place our result in a broader context.

The proof of exponential decay for ordinary (independent) two-dimensional percolation goes back to the celebrated paper \cite{Ke80} by Kesten.
A crucial step in that paper is, somewhat informally and in ``modern'' terminology, that if  the probability of the
event $A$ that  there is an occupied crossing of a given, large, box (square)  is neither
close to $0$ nor close to $1$, the expected number of
so-called pivotal vertices (or, for bond percolation, pivotal edges) is large.
(These are vertices with the property that flipping the state of the vertex flips the occurrence/nonoccurrence of the event $A$.)
This step was proved in a ``constructive'' way with a ``geometric'' flavor.
The above-mentioned large expectation of pivotal vertices implies that the derivative (w.r.t. the parameter $p$) of the probability of $A$
is large. Hence, once the probability of $A$ is not very small, a small increase of $p$ makes it close to $1$.
This property would now be called a ``sharp-threshold'' phenomenon.

Moreover, by separate arguments, so-called finite-size criteria hold: if the probability of $A$ is smaller than some absolute constant
$\epsilon$, the cluster size is finite a.s. (and its distribution has exponential decay), while if it is larger than $1 - \epsilon$ the
system percolates. Combining these things gives exponential decay of the cluster size for all $p$ smaller than $p_c$.

Russo \cite{Ru82} proved a very general ``approximate zero-one law'' and showed that the above mentioned sharp-threshold phenomenon
can be obtained from this more general law using only a minimum of percolation arguments. In this way Kesten's
``constructive, geometric'' arguments could be avoided, which is very useful because carrying out such arguments turns out to be
(too) hard in many dependent models. We should note, however, that for independent percolation the ``constructive'' argument still gives
the shortest self-contained proof and that in {\em some} dependent models
(see \cite{BaCaMe09}) it gives the only currently known proof.

Unfortunately, the above-mentioned finite-size criteria involved a so-called RSW result of which no (``reasonably general'') extension to
dependent models was known. This explains why for a long time Russo's approximate zero-one law did not receive much attention in the
percolation community.
In the meantime sharper and more explicit results related to Russo's approximate zero-one law were obtained (in other
areas of probability and mathematics in general) by
Kahn, Kalai and Linial \cite{KaKaLi88},  Talagrand \cite{Ta94} and Friedgut and Kalai \cite{FrKa96}. (See also
\cite{BKKKL} and \cite{Ro08}.)

The importance for percolation of these sharp-threshold results became clear much later when Bollob\'as and Riordan \cite{BoRi06a} proved a more robust
version of the RSW theorem which, combined with a clever use of the sharp-threshold results, led to the proof of the long-standing
conjecture that the critical probability for random Voronoi percolation in the plane is $1/2$ (and that below $1/2$ this model has exponential
decay). The robustness of these arguments led to similar results for several other two-dimensional percolation
models (see \cite{BoRi06b,Be08,BoRi08}).

The last-mentioned paper proved for 2D lattice models exponential decay below the percolation threshold under the quite general
condition that, informally speaking, the
model has a ``nice finitary representation'' (in a well-defined sense) in terms of finite-valued independent random variables
(see also \cite{BeSt}).
It turned out that under that condition only a weak (not explicitly quantitative) form, close to that of Russo's \cite{Ru82},
of the sharp-threshold results was needed.
As an example it was shown that the Ising model (with fixed $\beta < \beta_c$ and external field parameter~$h$ playing the role of $p$ in
ordinary percolation) belongs to this class thus giving an alternative, more streamlined proof of the main result in
Higuchi's paper \cite{Hi93}. Here the role of finite-valued independent random variables was played by the ``independent
updates'' in a suitable discrete-time dynamics. Such a dynamics was possible by (among other things) the nearest-neighbor Gibbs property of
the Ising model.

This is a big difference with the contact process for which we do not know a
suitable discrete-time dynamics. Therefore, we are not able to derive exponential decay for this model from
Theorem 2.2 in \cite{Be08} but instead exploit the full quantitative nature of the sharp-threshold results
from \cite{KaKaLi88} and \cite{Ta94} and follow more closely the route
used in \cite{BoRi06a} and \cite{BoRi08}
for the Voronoi model and the Johnson--Mehl model (which, like the Voronoi model, is a model of planar tessellations but more
complicated than the Voronoi model). Yet another route, namely by using results in \cite{GrGr06}, might work if
$\bnu$ would satisfy the strong FKG condition which, however (as has been shown by Liggett), it does not. We should
also note here that
the exponential-decay arguments in \cite{AiBa87} and \cite{Me86}, which for ordinary percolation work in all dimensions, so far have (even in 2D)
no suitable analog for dependent percolation.

Our main result is the following theorem.

\begin{theorem}\label{mainthm}
Let $\la$ be such that
\[
\lbnu(|\calC| = \infty) = 0.
\]
Then, for every $\lambda' < \la$ there exist $C_1, C_2 > 0$ such that for all $n \geq 1$

%e1 ###
\begin{equation}
\label{exp}
\bnu_{\la'}( |\calC| \geq n) \leq C_1 \exp(- C_2 n).
\end{equation}

\end{theorem}

Section~\ref{sec2} states properties of the contact process and other more general ingredients needed in the proof.
It also indicates (see the Remark below the proof of Lemma \ref{lem-fs}) an alternative proof of the earlier-mentioned
result by Liggett and Steif that percolation occurs for $\la$ large enough.

The proof of Theorem \ref{mainthm} is given in Section~\ref{sec3}. As mentioned before, the essence is still (as it was in \cite{Ke80}) to
show sharp-threshold behavior for certain crossing probabilities. To do this we
follow the main strategy in \cite{BoRi06a} and \cite{BoRi08}. However, the model-specific properties of the contact process
lead to many nontrivial differences in the steps. Therefore, and because the contact process is one of the main
random spatial models, the proof is given in detail.

We use several well-known results, techniques and terminology from percolation theory. For an introduction to, and general information on,
percolation see \cite{Gr99} and \cite{BoRi06c} and contact processes see \cite{Li85} and \cite{Li99}.

Throughout this paper we use the notation $V \subset\subset W$ to express that $V$ is a finite subset of $W$.

%s2 ###
\section{Preliminaries}\label{sec2}

%s2.1 ###
\subsection{Contact process ingredients}\label{sec2.1}

A well-known classical result for the contact process is that there is a critical value
$\lambda_c$ such that:
\begin{longlist}[(a)]
\item[(a)] If $\lambda < \lambda_c$ the contact process ``dies out'' and $\bnu$ is concentrated on the trivial configuration
where all vertices are vacant.
\item[(b)] If $\lambda > \lambda_c$, $\bnu$ is nontrivial and $\mu_t$ converges exponentially to $\bnu$ as
$t \rightarrow \infty$ (see \cite{Li99}, Theorem 2.30 and equation (2.31), which are based on the work by Bezuidenhout and Grimmett
\cite{BeGr90,BeGr91}):
For all $\lambda > \lambda_c$ there exist $C_3, C_4 > 0$ such that for all
$t > 0$
%
%e2 ###
\begin{equation}
\label{expt}
\mu_t(\sigma_0 = 1) - \bnu(\sigma_0 = 1) \leq C_3 \exp(- C_4 t).
\end{equation}
\end{longlist}

Since $\bnu$ is dominated by $\mu_t$, statement (b) above implies by standard arguments:

\begin{theorem}\label{BeGr}
For all $\lambda > \lambda_c$ there exist $C_3, C_4 > 0$ such that for
all $t > 0$ and all $\Lambda \subset \subset \Z^2$
%
%e3 ###
\begin{equation}
\label{expt2}
d_{\mathrm{TV}}(\mu_{t; \Lambda}, \bnu_{\lambda; \Lambda}) \leq |\Lambda| C_3 \exp(- C_4 t),
\end{equation}
where $d_{\mathrm{TV}}$ denotes variational distance
and $\mu_{t; \Lambda}$ (and $\bnu_{\lambda; \Lambda}$) are the restriction of
$\mu_t$ (resp. $\bnu_{\la}$) to $\La$.
\end{theorem}

\begin{rem*} It is trivial from the definition of $\la_c$ that for $\la$ below $p_c$ no percolation of occupied vertices occurs,
that is,
${\bnu}_{\la}(|\calC| = \infty) = 0$.
As we mentioned in the \hyperref[sec1]{Introduction}, Liggett and Steif \cite{LiSt06} showed that if $\la$ is large enough percolation does occur. It seems to be
widely believed (but no
proof is known yet) that the critical value for having percolation is strictly larger than $\la_c$.
(See \cite{LiSt06} where this problem is formulated.)
\end{rem*}

A well known and very useful way to describe the contact process is by means of a space--time diagram or graphical
representation (see, e.g., \cite{Li85} for historical background and references).
Consider for each vertex $v \in \Z^2$ its ``time axis'' $\{v\} \times (-\infty, \infty)$
and consider five independent Poisson point processes on this time axis: one with rate $\la$ for each of the
four directions (left, right, up, down) in the lattice
and one to indicate a transition from $1$ to $0$. The Poisson processes of the different vertices are independent of each other.

The interpretation of a  Poisson point on the time axis of $v$ at time $t$ for (say)
the direction ``right'' is that if $v$ is  in state $1$ at time $t$, it ``infects'' the vertex $v + (1,0)$. That is, if
the latter vertex is not occupied, it becomes occupied. To visualize this
we draw an arrow from $(v,t)$ to $(v +(1,0), t)$. We say that $t$ is the time coordinate of the arrow.
For each of the other three directions we act similarly. The interpretation of a
Poisson point in the fifth process on the time axis of $v$ at time $t$ is that if $v$ is occupied (i.e., in state $1$) at time $t^-$, it becomes
immediately vacant~($0$). In the space--time picture this is marked by the symbol $*$ at $(v,t)$
(see, e.g., \cite{Li99}, Part I, Section 1).

An active space--time path is a path that is allowed to move upward in time along the time axes without hitting $*$ points
and to jump from one time axis to another along, and in the direction of, an arrow. The time coordinates of the arrows
followed by a space--time path will be called the jumping times of the path. For $v, w \in \Z^2$ and $s < t$ we
denote by $(v,s) \rightarrow (w,t)$ that
there is an active path from $(v,s)$ to $(w,t)$. For the contact process starting at time $0$ with every vertex occupied,
a vertex $w$ is occupied at time $t > 0$ if and only if (in terms of the above-mentioned space--time diagram) for some vertex $v$
there is an active path from $(v,0)$ to $(w,t)$. In other words, the joint distribution of the random variables
\[
I\{\exists v \in \Z^2 \mbox{ s.t. } (v,0) \rightarrow (w,t) \}, \qquad   w \in
\Z^2
\]
is $\mu_t$. Similarly, $\bnu$ is the joint distribution of the random variables
\[
I\{\forall t < 0   \exists v \in \Z^2 \mbox{ s.t. } (v,t) \rightarrow (w,0) \}, \qquad   w \in \Z^2.
\]

We will often work with the following ``truncated'' random variables.
First some more notation: The distance between two vertices $v = (i_1, j_1)$ and $w = (i_2,j_2)$ is defined as
$\max(|i_1-i_2|, |j_1-j_2|)$ and denoted by $d(v,w)$. The distance $d(V,W)$ between two subsets $V$ and $W$ of $\Z^2$ is defined as
$\min(\{d(v,w)  \dvtx  v \in V, w \in W\})$. For $\La \subset \Z^2$, $\si_{\La}$ denotes the collection of random variables
$(\si_v, v \in \La)$; straightforward generalizations of this notation will also be used.\\
Let
%e4 ###
\begin{equation}
\label{truncdef}
\quad \sixn := I\bigl\{\exists (y,t) \mbox{ with } d(x,y) = \bigl\lfloor \sqrt n \bigr\rfloor \mbox{ or } t = -\sqrt n \mbox{ s.t. } (y,t) \rightarrow
(x,0)\bigr\}
\end{equation}
and let ${\bnu}^{(n)} = {\bnu}^{(n)}_{\la}$ denote the joint distribution of the random variables $\sixn, x\in \Z^2$.

It is clear from this definition that if $\La$ and $\La'$ are two finite subsets of $\Z^2$ and
$d(\La, \La') > 2 \sqrt n$, then $\si_{\La}^{(n)}$ and ${\si}_{\La'}^{(n)}$ are independent. It is also
clear that $\si$ is stochastically dominated by $\si^{(n)}$.

From Theorem \ref{BeGr}, and simple estimates concerning the ``spatial spread of infection in a limited time interval,''
it follows that
\begin{eqnarray}
\label{expn}
  &&\forall \la > \la_c    \exists C_5, C_6 > 0 \mbox{ s.t. }
\forall \La \subset \subset \Z^2\nonumber
\\[-8pt]\\[-8pt]
  &&\qquad d_{\mathrm{TV}}\bigl(\si_\La, \si_\La^{(n)}\bigr) \leq |\La|   C_5 \exp(- C_6 n^{1/2}).\nonumber
\end{eqnarray}

\begin{rem*} In this paper we often deal with spatial boxes of length of order $n$ and distances of order $n$ to each other. The
somewhat arbitrary choice of $\sqrt n$ in the definition \eqref{truncdef} is just one of the many possible choices that are
convenient in such situations.
\end{rem*}

\begin{lem} \label{mix-lem}
Let $\La_1, \ldots, \La_k$, be $3 n \times n$-rectangles with the property that $d(\La_i, \La_j) > 2 \lfloor \sqrt n \rfloor$,
$1 \leq i < j \leq k$. Further let $A_1, A_2, \ldots, A_k$ be events\break that are completely determined by, and increasing in, the $\si$ variables on
$\La_1, \La_2, \ldots, \La_k$ respectively.
Then, for every $\la > \la_c$,
\begin{eqnarray}
\label{box-mix}
\prod_{i=1}^k {\bnu}_{\la}(A_i)  &\leq&  {\bnu}_{\la} \Biggl(\bigcap_{i = 1}^k A_i\Biggr) \leq
{\bnu}_{\la}^{(n)} \Biggl(\bigcap_{i = 1}^k A_i\Biggr)\nonumber
\\[-8pt]\\[-8pt]
  &=& \prod_{i=1}^k {\bnu}_{\la}^{(n)}(A_i) \leq \prod_{i = 1}^k \bigl({\bnu}_{\la}(A_i) + C_5 8 n^2 \exp\bigl(-C_6 \sqrt n\bigr)\bigr).\nonumber
\end{eqnarray}
\end{lem}

\begin{pf}
The first inequality comes from the well-known positive association of ${\bnu}_{\la}$ (which goes back to Harris's inequality)
and the last inequality comes from~\eqref{expn}.
The second inequality and the equality follow immediately from the definitions.
\end{pf}

Let, for a rectangular box $R$ in the lattice, $H(R)$ denote the event that there is an occupied horizontal crossing of $R$.
Further, let $H(n,m)$ denote the event that there is an occupied horizontal crossing of the box $[0,n] \times [0,m]$.
For vertical occupied crossings we use a similar notation, with $V$ instead of $H$.
From now on when we write ``crossing'' we always mean ``occupied crossing.''

\begin{lem}[(Finite-size criterion)]
\label{lem-fs}
\[
\exists \hep > 0,\qquad    \forall \la > \la_c,\qquad    \exists \hN,\qquad    \forall N \geq \hN
\]
the following holds:
\begin{longlist}[(a)]
\item[(a)]
%e5 ###
\begin{equation} \label{fs1}
\mbox{If } \lbnu(V(3 N, N)) < \hep, \mbox{ the distribution of } |\calC| \mbox{ has exponential decay.}
\end{equation}
\item[(b)]
%e6 ###
 \begin{equation}\label{fs2}
\mbox{If } \lbnu(H(3 N, N)) > 1- \hep, \mbox{ then } \lbnu(|\calC| = \infty) > 0.
\end{equation}
\end{longlist}
\end{lem}

\begin{pf}
The analog of part (a) was proved for ordinary percolation by Kesten in \cite{Ke81} by a block argument.
His proof can be, and has been in the literature, easily adapted to dependent models with sufficient spatial mixing (e.g., see
\cite{Be08}, Lemma 3.8). The mixing property described by \eqref{box-mix} above is more than enough
for this purpose.  Essential is that the ``extra term'' [here $C_5 8 n^2 \exp(-C_6 \sqrt n)$] in the factors in the right-hand side of
 \eqref{box-mix} goes to
$0$ as $n \rightarrow \infty$.

The analog of (b) was proved for ordinary percolation in \cite{ChCh84} by giving a suitable (and now well known)
lower bound for the probability of having a horizontal crossing of a $4 n \times 2n$ box
in terms of the probability of the analogous event for a $2 n \times n$ box.
If for some $n$ this probability is sufficiently close to $1$, one can then iterate this procedure and
conclude that the probability, say $r_k$, of a crossing of a given $2^{k+1} n \times 2^k n$ box goes very fast to $1$
as $k \rightarrow \infty$. (So
fast that $\sum_k (1-r_k)$ is finite.) By Borel--Cantelli it then follows that a.s. there is a $K$, such that for all odd $k \geq K$,
there is horizontal crossing of the rectangle $[0, 2^{k+1} n] \times [0, 2^k n]$ and for all even $k \geq K$
there is a vertical crossing of $[0, 2^k n] \times [0, 2^{k+1} n]$. By pasting together these crossings one gets an infinite
occupied path.
Hence, the system percolates. For dependent percolation models with sufficiently strong mixing properties
simple modifications of such arguments can be obtained (and have been obtained in the literature). Informally speaking, instead of blowing the
rectangles up by a factor $2$, this
is then done by a factor $3$ to obtain an extra strip in the middle of the next rectangle in order to separate the other two strips so that
the crossing events of these other two strips are almost independent. See, for instance, \cite{BeBrVa08}, proof of Theorem 4.8, for a case
where this has been carried out in detail. In practically the same way this can be carried out in our current situation by using Lemma \ref{mix-lem}
above in the same way as Lemma 2.3 was used in \cite{BeBrVa08}, proof of Theorem 5.1.
\end{pf}

\begin{rem*} For our purpose (as will become clear later in this paper) we do not need $\hat N$ in
Lemma \ref{lem-fs} to be uniform in $\la$ if
$\la$ is bounded away from $\la_c$. However, although this is not explicitly stated in the literature but pointed out to me by Geoffrey
Grimmett  (private communication), \eqref{expt} and related bounds are, by the nature of their proofs in the literature, uniform in
$\la$, if $\la$ is bounded away from $\la_c$. Now such uniformity would also give
uniformity of $\hat N$, in the sense mentioned above. This then, in turn, would clearly give an alternative proof of the earlier
mentioned result by Liggett
and Steif that $\bnu_{\la}$ has percolation if $\la$ is large enough: Take some $\la' > \la_c$. Fix $N$ such that for all $\la > \la'$
the ``if-then statement'' (b) in Lemma \ref{lem-fs} holds.
It is easy to see that, with $N$ fixed, if $\la > \la'$ is large enough, the {\it condition} in that ``if-then statement''
(b) holds; hence
$\lbnu(|\calC_O| = \infty) > 0$. Since this result is already known and not the main subject of this paper, we do not work out the details of
such alternative proof. It should also be noted that Liggett and Steif prove more than percolation of $\bnu$. They show,
for large~$\la$, domination of high-density product measures.
\end{rem*}

The following involves what in the \hyperref[sec1]{Introduction} was called a robust version of RSW.

\begin{prop} \label{prop-RSW}
Let $\la > \la_c$. If
\[
\mbox{for {\em some} } \rho > 0 \qquad \limsup_{n \rightarrow \infty} \lbnu(H(\rho n, n)) > 0,
\]
then
\[
\mbox{for {\em all} } \rho > 0\qquad \limsup_{n \rightarrow \infty} \lbnu(H(\rho n, n)) > 0.
\]
\end{prop}

\begin{pf}
A similar result was proved by Bollob\'as and Riordan \cite{BoRi06a} for the random Voronoi model (and slightly modified to the above
form in \cite{BeBrVa08}).
As remarked in \cite{BoRi06b} (see also \cite{Be08}, Section 3.4, the first three paragraphs) it holds for many
percolation models on $\Z^2$, namely, those that
satisfy: (i) a sufficiently strong mixing
property, (ii) a straightforward ``geometric'' condition
about lattice paths (which enables pasting together paths that cross each other), (iii) positive association and (iv) the condition that $\bnu$
is invariant under the symmetries of $\Z^2$.

Lemma \ref{mix-lem} above is more than needed for (i) and it is easy to see that the probability measures ${\bnu}_{\la}$, $\la > \la_c$ also
satisfies the other conditions.
\end{pf}

%s2.2 ###
\subsection{Influence and sharp-threshold results}
\label{Sharp-threshold results}

Let $\Om = \{0,1\}^n$ and let $P_p$ denote the product measure with parameter $p$ on $\Om$.
Let $A$ be an event (i.e., a subset of $\Om$) and let, for $1 \leq i \leq n$, $I_i$ denote the probability that
$i$ is pivotal for $A$. It is often called the influence of $i$. More precisely,
\[
I_i := P_p\bigl(\bigl\{\om \in \Om   \dvtx   \mbox{ exactly one of } \om \mbox{ and } \om^{(i)} \mbox{ is in } A\bigr\}\bigr),
\]
where $\om^{(i)}$ is the configuration obtained from $\om$ by flipping the $i$th component of~$\om$.
Talagrand (\cite{Ta94}, Corollary 1.2) proved the following theorem.   See also \cite{FrKa96} and \cite{KaKaLi88} for strongly related results.
Note that our $I_i$ differs a factor $1/p$ from
the expression $\mu_p(A_i)$ in Talagrand's paper.

\begin{theorem}
\label{tal}
%e7 ###
\begin{equation} \label{tal-eq}
\sum_i I_i \geq \frac{P_p(A) (1 - P_p(A))}{K p \log(2/p)} \log\biggl(\frac{1}{p \max_i I_i}\biggr),
\end{equation}
where $K$ is a universal positive constant.
\end{theorem}

\begin{rem*}
\begin{longlist}[(ii)]
\item[(i)] Strictly speaking Talagrand's result is slightly stronger than Theorem~\ref{tal} above
but in the case of small $p$ (to which we will apply it), it makes essentially no difference.

\item[(ii)] If the event $A$ is increasing (i.e., its indicator function is a coordinate-wise nondecreasing function on $\Om$),
the left-hand side of \eqref{tal-eq} is, according to Russo's formula, equal to $d / dp   P_p(A)$.
By this it is easy to see that Theorem~\ref{tal} implies that if, throughout some interval, say $(p_1, p_2)$,
$\max_i I_i$ is ``very small'' and $P_{p_1}(A)$ is ``not too small,'' then $P_{p_2}(A)$ is ``close to $1$.''
For such reasons Theorem~\ref{tal} and related theorems are often indicated by the name ``sharp-threshold'' results, in addition to names like
``influence results.''
\end{longlist}
\end{rem*}

Now suppose there are at least $m$ indices $i$ with the property that $I_i = \max_j I_j$.
There are two possibilities:

\begin{longlist}[(a)]
\item[(a)] $\max_i I_i \leq \frac{\log m}{p m}.$ If this holds then, by Theorem \ref{tal},
%e8 ###
\begin{equation}\label{talm-eq}
\quad \sum_i I_i \geq \frac{P_p(A) (1 - P_p(A))}{K p \log(2/p)} \log\biggl(\frac{m}{\log m}\biggr) \geq \frac{P_p(A) (1 - P_p(A))}
{\tilde K p \log(2/p)} \log m
\end{equation}
for some universal constant $\tilde K$.

\item[(b)] $\max_i I_i \geq \frac{\log m}{p m}$.
Then trivially,
\[
\sum_i I_i \geq m \max_i I_i \geq \frac{\log m}{p}
\]
which is larger than or equal to some universal constant times the right-hand side of \eqref{talm-eq}. Hence, by adjusting
the value $K$ if needed, the following holds:
\end{longlist}

\begin{cor}
\label{talm-cor}
Let $m$ denote the cardinality of $\{i   \dvtx   I_i = \max_j I_j\}$. Then
\[
\sum_i I_i \geq \frac{P_p(A) (1 - P_p(A))}{K p \log(2/p)} \log m.
\]
\end{cor}

\begin{rem*} The case $m = n$ of this corollary is essentially in \cite{FrKa96} where it is derived from the results/methods in \cite{KaKaLi88}.
The general case, and its derivation from Theorem \ref{tal}, was shown to me by Oliver Riordan (private communication; see also \cite{BoRi09}).
\end{rem*}

We will use a generalization of Theorem \ref{tal} and Corollary \ref{talm-cor} as described below.

Let $\Om$ be as before. Let $V \subset \{1, \ldots, n\}$ and let $0 < p_1, p_2 <1$. Let $P_{p_1, p_2}$ denote
the product measure on $\Om$
under which each component with index in $V$ is $1$ with probability $p_1$ and each with index in $V^c$
is $1$ with probability $p_2$. The generalization of Theorem \ref{tal} is the
following theorem.

\begin{theorem}
\label{tal-gen}
\[
 \sum_i I_i \geq
\frac{P_{p_1,p_2}(A) (1 - P_{p_1,p_2}(A))}{K'   \max(p_1, p_2) \log(2/\min(p_1, p_2))}
\log\biggl(\frac{1}{\max(p_1, p_2) \max_i I_i}\biggr),
\]
where $K'$ is a universal constant.
\end{theorem}

\begin{rem*} In \cite{BoRi09} (see Theorem 5 in \cite{BoRi09} and the discussion below that theorem) it is indicated how
to prove Theorem \ref{tal-gen} by modifications of the proofs in Talagrand's paper. An alternative way is to start from
the special case for $p = 1/2$ of Theorem \ref{tal} above and obtain the full case (and its generalization where
different coordinates may have a
different parameter $p$) from that special case by, informally speaking, representing (approximately) the toss of a
biased coin by a combination of tosses of several fair coins.
\end{rem*}

From Theorem \ref{tal-gen} the following corollary is obtained in exactly the same way as Corollary \ref{talm-cor} was
obtained from Theorem \ref{tal}.

\begin{cor}
\label{talm-cor-gen}
Let $m$ denote the cardinality of $\{i   :   I_i = \max_j I_j\}$. Then
\[
\sum_i I_i \geq
\frac{P_{p_1,p_2}(A) (1 - P_{p_1,p_2}(A))}{K'  \max(p_1, p_2)   \log(2/(\min(p_1, p_2)))} \log m.
\]
\end{cor}

Combined with a straightforward modification of the earlier-mentioned Russo's formula this gives:

\begin{cor}
\label{talm-cor-dif}
Let $m$ be as in the previous corollary. If the event $A$ is increasing in the coordinates with parameter $p_1$ and
decreasing in the coordinates with parameter $p_2$, then
\begin{eqnarray} \label{talm-cor-dif-eq}
&&\frac{\partial}{\partial p_1} P_{p_1, p_2}(A) - \frac{\partial}{\partial p_2} P_{p_1, p_2}(A)\nonumber
\\[-8pt]\\[-8pt]
&&\qquad\geq \frac{P_{p_1,p_2}(A) (1 - P_{p_1,p_2}(A))}{K'  \max(p_1, p_2)   \log(2/(\min(p_1, p_2))} \log m.\nonumber
\end{eqnarray}
\end{cor}

%s3 ###
\section[Proof of Theorem 1.1]{Proof of Theorem \protect\ref{mainthm}}\label{sec3}

Let $\la_1 > \la_c$ be such that under ${\1lbnu}$ the cluster size distribution does not have exponential decay. Let
$\la_2 > \la_1$. We will show that
${\bnu}_{\la_2}(|\calC_O| = \infty) > 0$.
This will immediately imply Theorem \ref{mainthm}.

Let $\la_1$ be as fixed above and let $\hep$ and $\hN = \hN(\la_1)$ be as in Lemma \ref{lem-fs}.
Let $L_n$ denote a specific $4 n \times n$ rectangle; its precise choice does not matter but for later convenience we
choose $[n, 5 n] \times [n, 2 n]$.
By Lemma \ref{lem-fs} we have that
\[
 \1lbnu(V(3n,n)) > \hep\qquad \mbox{for all } n \geq \hN
\]
which by Proposition \ref{prop-RSW} implies $\limsup_{n \rightarrow \infty} \1lbnu(H(L_n)) > 0$;
so there exists an $\tep > 0$ and a sequence $n_1, n_2, \ldots$ such that
%
%e9 ###
\begin{equation}
\label{lbLcr}
\1lbnu(H(L_{n_i})) > \tep\qquad    \mbox{for all } i.
\end{equation}
From now on we consider such fixed sequence.

In the \hyperref[sec1]{Introduction} to the contact process in the beginning of Section~\ref{sec1} we assumed that the recovery rate is $1$.
Of course the contact process
with infection rate~$\la$ and recovery rate $\de$ is simply a time-rescaled version of the contact process
with infection rate $\la/\de$ and recovery rate $1$. In particular, these two contact processes have exactly
the same upper invariant measure.
For application of the results in Section \ref{Sharp-threshold results} it is more convenient to work
with one-parameter Poisson processes for which at each site of the lattice the
total rate of all the Poisson processes is constant, say $1$. Therefore, we
consider the contact process with infection
rate $q/4$ and recovery rate $1-q$, where now $q \in (0,1)$ is the parameter.
Note that in terms of the space--time diagram this means that on each time axis
we have a marked Poisson point process with density $1$ and each point corresponds with a $\rightarrow$, $\leftarrow$,
$\downarrow$, $\uparrow$ or $*$ with probability $q/4$, $q/4$, $q/4$, $q/4$ and $1-q$, respectively.
With respect to this new parameter $q$ we use the notation $\calP_q$ for the law governing the above-marked Poisson point process and
the notation $\qbnu$ for the upper invariant measure of the corresponding contact process.
From the above it is immediate that
%e10 ###
\begin{equation}
\label{q-la-trans}
\qbnu = \bnu_{{q/(4(1-q))}}, \qquad   q \in (0,1),
\end{equation}
or, equivalently, $\lbnu = \bnu_{\langle 4 \lambda / (1+4 \lambda)\rangle}$, for $\la \in (0, \infty)$.
In particular, by \eqref{lbLcr},
%e11 ###
\begin{equation}
\label{lbQcr}
{\bnu}_{\langle q_1\rangle}(H(L_{n_i})) > \tep  \qquad  \mbox{ for all } i,
\end{equation}
where $q_1 = 4 \la_1 / (1 + 4 \la_1)$.

Let $\qnbnu$ be the distribution of $(\et_x^{(n)},   x \in \Z^d)$ defined by [compare with \eqref{truncdef}]
\begin{eqnarray}
\label{truncdefsig}
\etxn &:=& I\bigl\{\exists (y,t) \mbox{ with } d(x,y) = \bigl\lfloor \sqrt n \bigr\rfloor \mbox{ or }\nonumber
\\[-8pt]\\[-8pt]
&&\hspace*{11pt}t < -\sqrt n \mbox{ s.t. } (y,t)
\stackrel{(q, 1 - q)}{\rightarrow} (x,0)\bigr\},\nonumber
\end{eqnarray}
where $(y,t) \stackrel{(q, 1 - q)}{\rightarrow} (x,0)$ denotes that there is a space--time path from
$(y,t)$ to $(x,0)$ in the space--time diagram with Poisson intensity $q/4$ for each of the four types of arrows
and Poisson intensity $1-q$ for $*$'s.

It is clear that $\qnbnu$ dominates $\qbnu$; hence, by \eqref{lbQcr},
%
%e12 ###
\begin{equation}
\label{lbLcr2sig}
\bnu_{\langle q_1\rangle}^{(n_i)}(H(L_{n_i})) > \tep\qquad    \mbox{ for all } i.
\end{equation}

Although $\lnbnu$ is, of course, not the same as $\bnu_{\langle 4 \lambda/(1+4 \lambda)\rangle}^{(n)}$,
it is straightforward to get analogs of the earlier ``approximation lemmas.'' In particular we get, as an analog of \eqref{expn}, \\
\begin{eqnarray}
\label{expnsig}
  &&\forall q > 4 \la_c/(1 + 4 \la_c)   \exists C_7, C_8 > 0 \mbox{ s.t. }\nonumber
\forall \La \subset \subset \Z^2
\\[-8pt]\\[-8pt]
  &&\qquad d_{\mathrm{TV}}\bigl(\bnu_{\langle q\rangle; \La}, \bnu_{\langle q\rangle; \La}^{(n)}\bigr) \leq |\La|   C_7 \exp(- C_8 n^{1/2}).\nonumber
\end{eqnarray}
Throughout the proof of Theorem \ref{mainthm}, except at the very end (see Proposition~\ref{H-bound}, where we translate back
to parameter $\la$), we will work with parameter $q$ as described above.

A key step toward application of the results in Section \ref{Sharp-threshold results} is a suitable ``time-discretized''
version of $\qnbnu$. A significant obstacle is to obtain an analog of \eqref{lbLcr2sig} for these discrete variables.

Recall from the beginning of this section that $L_n$ is the box $[n,5 n] \times [n, 2 n]$. To ``get ample room for the underlying
Poisson points'' we also consider the larger box $B_n : = [0,6 n] \times [0, 3 n]$.
Let $\qnbnu$ be as before.
Note that the collection\vspace*{-2pt} of random variables $(\etxn,   x \in L_n)$ is completely determined by the (marked)
Poisson points in the space--time area
$\mathit{ST}(n) := B_n \times [-n,0]$. (In fact only a subset of that area is involved but for
convenience we consider this whole area.)
Let as before, $\calP_{q}$ denote the probability measure governing the marked Poisson points.

Let $0 < \al < 1$. Later we choose $\al$ sufficiently small. Let $\delta = n^{-\al}$.

\begin{defn}\label{stab-def}
We say that an active space--time path $\pi$ is $\delta$-stable if the following hold:
\begin{longlist}[(ii)]
\item[(i)] If $s$ and $t$ are two different jump times of $\pi$, then $|t - s| > \delta$.

\item[(ii)] If $(y,s)$ is the starting point or endpoint of an arrow of $\pi$ and there is a $*$ at $(y,t)$, then
$|t - s| > \delta$.
\end{longlist}
\end{defn}

The following lemma (and the global structure of its proof) is the analog of Theorem 6.1 for the Voronoi model in \cite{BoRi06a} and
Theorem 8 for the Johnson--Mehl tessellations in \cite{BoRi08} (see also \cite{BoRi09}). Since the proof is subtle and differs in many
details from that in \cite{BoRi06a} and \cite{BoRi08} we give a full proof.

\begin{rem*} In some sense the proof of Lemma \ref{stabco} is easier and shorter than that of the corresponding results in \cite{BoRi06a}
and \cite{BoRi08}. This is partly due to the fact that in our model the continuous object that has to be properly discretized (the time axis)
is one dimensional. This enables us to ``play'' with the order (in time) of the Poisson points. On the other hand, our model has some
extra complications, for example, there is no natural order on the arrow values assigned to the Poisson points (an arrow to the right is not always
better than an arrow to the left). Fortunately these issues can be handled quite smoothly.
\end{rem*}

\begin{lem}[(Stability coupling)]
\label{stabco}
Let $0 < q < q' < 1$. For each $n$ there is a coupling of
$\calP_{q}$ and $\calP_{q'}$ such that w.h.p. (i.e., with probability tending to $1$ as $n \rightarrow \infty$) the
following holds: For every $x \in L_n$ that has $\etxn = 1$ in
the first copy, there is a $(y,t) \in \Z^2 \times (-\infty,0)$ with $d(x,y) = \lfloor \sqrt n \rfloor$ or
$t = -\sqrt n$ such that there is a $\delta$-stable space--time path
in the second copy from  $(y,t)$ to $(x,0)$ and hence $\etxn$ also equals $1$ in the second copy.
\end{lem}

\begin{pf}
Let $\de_1 = n^{-\al/2}$. So $\de \ll \de_1$.
We partition every
``time axis'' $\{x\} \times [-\infty,0]$, $x \in \Z^2$, in intervals
$\{x\} \times (-(k+1) \de_1, -k \de_1],   k = 0, 1, \ldots,$ of length $\de_1$.
From now on when we use the word ``interval,'' we will always mean an interval of the above form with $x \in B_n$ and
$(k+1) \de_1 \leq n$.
Note that the total number of intervals is $M_n := |B_n| \lfloor n/\de_1 \rfloor$.
Let ${\mathcal I}_n$ denote the union of these intervals.

Note that the total number of Poisson points in ${\mathcal I}_n$ is Poisson distributed with mean
$\de_1 M_n$.
To construct the coupling first draw a number $N$ according to the above-mentioned Poisson distribution.
Now assign $N$ points (called ``particles'') randomly, uniformly and independently of each other to the
above mentioned set~${\mathcal I}_n$. If a particle is assigned to the space--time location $(x,t)$, we say that its time coordinate is $t$.
Call an interval ``occupied'' if it has at least one particle.
Call two different intervals $\{x\} \times (-(k+1) \de_1, -k \de_1]$ and
$\{y\} \times (-(l+1) \de_1, -l \de_1]$ neighbors if $d(x,y)\leq 1$ and
$|k - l| \leq 1$.
This gives rise in an obvious way to the notion of ``clusters (of occupied intervals).''
(This notion of cluster is of course different from that introduced earlier in this paper.
Since this ``new'' notion of cluster is used only in this proof
and the other notion is not used here, this should not cause any confusion and we even use the same notation $\mathcal C$.)

We have already assigned to each
particle a precise location in ${\mathcal I}_n$.
However, we ``suppress'' this precise information and only ``keep'' the following {\em partial information}: for each interval the number of
particles assigned to it and for each occupied cluster of intervals the relative order (w.r.t. their time coordinates) of all particles in that cluster.
We also assign, with equal probabilities ($1/4$),
a {\em tentative} $\leftarrow$, $\rightarrow$, $\uparrow$ or $\downarrow$ to each particle (independent
of the other particles). The interpretation is that {\em if} eventually a particle is chosen to represent
an arrow, the type of arrow is exactly its above-mentioned tentative one.

\begin{rem*}
From now on when we mention a cluster $\mathcal C$, we
mean not only its corresponding set of intervals but also the above-mentioned partial information
about the particle locations as well as the tentative arrows assigned to the particles.
\end{rem*}

By the size of a cluster we mean the number of particles in the cluster.

\begin{clm*}
There is a constant $D = D(\al)$ such that
%e13 ###
\begin{equation} \label{claim-eq}
\lim_{n \rightarrow \infty} P\bigl(\exists \mbox{ an occupied cluster with size } \geq D(\al)\bigr) = 0.
\end{equation}
\end{clm*}

\begin{pf}
Let $D$ be a positive integer. If the occupied cluster of a given interval~$e$ has size $\geq D$ there is
a connected set of $D$ (not necessarily occupied) intervals, such that $e$ is one of these intervals and the
number of particles in the union of these intervals is $\geq D$. Note that the number of choices for $e$ is
$M_n \leq n^4$ (for $n$ sufficiently large) and that for each
choice of $e$ the number of possible connected sets of $D$ intervals is
smaller than or equal to some constant $C(D)$ which depends on $D$ only. Further, the number of particles
in the union of $D$ given intervals is Poisson distributed with mean $D \delta_1 = D n^{- \al/2}$. So
the probability that this number of particles is $\geq D$ is at most $(D n^{- \al/2})^D$. Hence, the
probability that there is an occupied cluster of size $\geq $$D$ is at most
\[
n^4 C(D) (D n^{- \al/2})^D.
\]
If we take $D = \lceil 9/ \al \rceil$, this probability goes indeed to $0$ as $n \rightarrow \infty$.
This proves the above claim.
\end{pf}

Let $\mathcal C$ be a cluster in the sense given in the remark above. Now consider for both parameter values, $q$ and $q'$, the
conditional distribution of the precise
configuration for $\mathcal C$, that is, the types ($*$, $\leftarrow$, $\rightarrow$, $\uparrow$ or $\downarrow$) and precise locations of all
particles in $\mathcal C$, given the partial information.
The two conditional distributions can be coupled by the following natural procedure which gives two ``typical realizations'' (one
for each of the two parameter values) of the precise configuration.
% (This will not yet give the coupling in the statement of the Lemma; however,
%by a modification of this coupling we will obtain the required one).

The first step in this procedure is to assign to each particle $i$,
independent of the other particles, a
random variable $U_i$ uniformly distributed on $(0,1)$. These variables will be used below to decide if a particle
corresponds with an arrow or with a $*$.

The next step is to go from relative order of positions to precise positions.
Consider the conditional distribution of the precise time coordinates of the particles of $\mathcal C$, given their (already known) relative
order in time and the intervals they are located in. Now simply assign the precise locations by drawing from this distribution.
Later we will refer to this procedure as the ``time assignment procedure.''

Note that both steps above are the same for both ``realizations,'' the one for parameter $q$ and the one with parameter $q'$.
However, the next and final step in which the types of the particles are
fully determined will take into account the parameter value: for each particle $i$ of $\mathcal C$ do the following:
If $U_i < q$, the type of $i$ in each of the two copies is equal to the earlier mentioned
{\em tentative} arrow. If $U_i \in (q,q')$, its type is $*$ in copy 1 and equal to the tentative arrow in copy 2.
If $U_i > q'$, the type is $*$ in both copies.

Now we have two realizations, say $\om_{\mathcal C}(1)$ and $\om_{\mathcal C}(2)$, and
it is easy to see that they are ``typical'' w.r.t. the two conditional distributions mentioned above (the first for
parameter $q$, the second for parameter $q'$). So we
indeed have a coupling
of these two conditional distributions. Also note that $\om_{\mathcal C}(2) \geq \om_{\mathcal C}(1)$ in
the sense that the particle locations are exactly the same and each particle in $\om_{\mathcal C}(1)$
that has an arrow-type has the same arrow-type in $\om_{\mathcal C}(2)$. Let this coupling
be denoted by $P_{\mathcal C}$.

Doing this for each cluster, independently of the other clusters, gives a natural coupling of the two probability measures
in the statement of the lemma. However, it is not yet what
we want. Although it satisfies the property between brackets at the end of the lemma, it does not necessarily satisfy the stability property in the lemma.
The coupling we do want is obtained as follows where we go back to the level of a given cluster $\mathcal C$.
Recall the two copies $\om_{\mathcal C}(1)$ and $\om_{\mathcal C}(2)$ above and their
joint distribution $P_{\mathcal C}$.
From $P_{\mathcal C}$ we will construct a modified distribution ${\tilde P}_{\mathcal C}$ of which the two marginal distributions are the same
as those of $P_{\mathcal C}$. To avoid an abundance of notation we will drop the subscript $\mathcal C$ from $\om_{\mathcal C}(1)$
and $\om_{\mathcal C}(2)$.

Recall the time assignment procedure in the second step of the construction of $P_{\mathcal C}$.
Let $B$ be the event that in $\om(1)$ [and hence, since the particle locations for $\om(1)$
and $\om(2)$ are the same, also in $\om(2)$] there
are two different particles in $\mathcal C$ whose time coordinates differ at most $\de$.
The probability of $B$ (or, more precisely, the conditional probability of $B$ given the partial information on $\mathcal C$)
is maximal if $\mathcal C$ consists of one interval only,
in which case it is less than or equal to\vspace*{-2pt}
${|\mathcal C|}^2 \frac{2 \delta}{\delta_1} = {|\mathcal C|}^2 2 n^{-\al/2},$
where $|\mathcal C|$ denotes the number of particles in $\mathcal C$;
so
%
%e14 ###
\begin{equation}
\label{pcb}
P_{\mathcal C}(B) \leq  {|\mathcal C|}^2 2 n^{-\al/2}.
\end{equation}

Recall the use of the variables $U_i$ in the determination of the types of the points.
Let $G$ be the event that
each particle in $\om(1)$ is of type $*$ and each particle in $\om(2)$ has an arrow type.
Note that this event happens if and only if
$U_i \in (q,q')$ for all particles $i$ in $\mathcal C$ so that we have
\[
P_{\mathcal C}(G) = (q'-q)^{|\mathcal C|}.
\]

By this and \eqref{pcb} we have (with $D = D(\al)$ as in the claim above)
%
%e15 ###
\begin{equation}
\label{BG-ineq}
P_{\mathcal C}(G) \geq P_{\mathcal C}(B)\qquad \mbox{if } |{\mathcal C}| \leq D
\end{equation}
and $n$ is sufficiently large.
From now on we assume in this proof that $n$ is indeed sufficiently large in this sense.

Now let $B'$ denote $B \setminus G$.
If $|{\mathcal C}| \leq D$ then by \eqref{BG-ineq}
there is a measurable subset $G' \subset G \setminus B$
and a 1--1 map $\psi \dvtx B' \rightarrow G'$ with the property that $\psi$ and
$\psi^{-1}$ are $P_{\mathcal C}$-preserving. To each pair
$(\om(1), \om(2)) \in B'$ this map assigns the pair
\[
\psi(\om(1),\om(2)) = \bigl((\psi(\om(1),\om(2)))(1),(\psi(\om(1),\om(2)))(2)\bigr).
\]
Now a modified coupling called ${\tilde P}_{\mathcal C}$ is obtained from $P_{\mathcal C}$ by
exchange between $B'$ and $G'$ of the second copy, using the map $\psi$ as follows. (Such
type of modification is called a ``cross-over'' in \cite{BoRi06a}.)
If ${|\mathcal C|} \geq D$ we simply take ${\tilde P}_{\mathcal C} = P_{\mathcal C}$.
Otherwise,
a typical pair $(\tilde\om(1), \tilde\om(2))$ under ${\tilde P}_{\mathcal C}$ is drawn as follows. First draw a pair
$(\om(1), \om(2))$ under $P_{\mathcal C}$. If $(\om(1),\om(2)) \in (B' \cup G')^{c}$, take $(\tilde\om(1), \tilde\om(2))$
equal to $(\om(1), \om(2))$.
If $(\om(1),\om(2)) \in B'$, take $\tilde\om(1) = \om(1)$  and $\tilde\om(2) = (\psi(\om(1),\om(2)))(2)$. Finally, if
$(\om(1),\om(2)) \in G'$ take $\tilde\om(1) = \om(1)$ and $\tilde\om(2) = (\psi^{-1}(\om(1),\om(2)))(2)$.
Since in all cases $\tilde\om(1) = \om(1)$, it is immediate that the first marginal of ${\tilde P}_{\mathcal C}$ is equal to that of
 $P_{\mathcal C}$.
A short inspection shows that also the second marginal of ${\tilde P}_{\mathcal C}$ is equal to that of $P_{\mathcal C}$.

Now the ``overall'' coupling of
$\calP_{q}$ and $\calP_{q'}$ announced in the statement of the lemma is obtained in a natural and straightforward way
by constructing the pair $({\tilde\om}_{\mathcal C}(1), {\tilde\om}_{\mathcal C}(2)$ for each cluster $\mathcal C$ separately, independently of the other clusters.

To check the required properties of this coupling first look again at one single cluster $\mathcal C$.
Suppose that $|{\mathcal C}| \leq D$.
Let $(\om(1),\om(2))$ and the corresponding pair $(\tilde\om(1),\tilde\om(2))$ be as above. So, in particular,
$\tilde\om(1) = \om(1)$. Suppose that
$\tilde\om(1)$ has a certain active space--time path $\pi$ within $\mathcal C$. Note that $\pi$ is also
an active space--time path for $\om(1)$ and [because $\om(2) \geq \om(1)$ in the sense mentioned earlier in this proof]
also for $\om(2)$. For our purpose we may assume
that $\pi$ is part of a path that guarantees for some $x \in L_n$, that $\eta_x^{(n)} = 1$
(see the statement of Lemma \ref{stabco}). Therefore, by considering a trajectory of this longer path between entering
and leaving the cluster, we may assume that $\pi$ starts at the bottom
of some interval and ends at the top of some interval.  We will show that $\tilde\om(2)$
has a $\delta$-stable space--time path $\tilde\pi$ that ``corresponds'' with $\pi$. More precisely, although the jump-times
of the path $\tilde\pi$ may differ a bit
from the corresponding jump times of $\pi$, it will start and end at the same space--time points as the beginning,
respectively end, of $\pi$.

First we assume that $\pi$ makes at least one jump.
Since $\om(1)$ has at least one arrow in $\mathcal C$, $(\om(1),\om(2))$ is not in $G$, so we have only the
following two possible cases:
\begin{longlist}[(ii)]
\item[(i)] If $(\om(1),\om(2)) \in B' = B \setminus G$, then its image under the map $\psi$ is in $G \setminus B$.
Hence, since the relative order and the tentative arrow types of all the particles are fixed and by the definition of $G$
no particle in $\tilde\om(2)$ has a $*$, there is indeed a natural
path $\tilde \pi$ in the configuration $\tilde\om(2)$ that corresponds with $\pi$. Moreover, by the
definition of $B^c$ no two particles in
$\tilde\om(2)$ have time coordinates that differ at most $\delta$ and hence $\tilde \pi$ is $\delta$-stable.

\item[(ii)] If $(\om(1),\om(2)) \in B^c \cap G^c$, we have $\tilde\om(2) = \om(2)$. From the definition of
$B^c$ it
follows that $\pi$ itself is $\delta$-stable so we can take $\tilde\pi$ equal to $\pi$.
\end{longlist}

Now suppose $\pi$ makes no jump. So $\pi$ is, in fact, the union of a finite number of consecutive intervals on
the time axis of a vertex.
Note that by definition of a cluster each of these intervals has
at least one particle. Hence, $(\om(1), \om(2))$ is not in $G$ because otherwise in the configuration
$\om(1)$ each of these intervals would
have a $*$ which contradicts the fact that $\pi$ is an active path. If it is not in $B$ either,
$\tilde\om(2) = \om(2)$ and we can simply take $\tilde\pi = \pi$.
Finally, if $(\om(1), \om(2))$ is in $B\setminus G$, then its image is in $G\setminus B$ so $\tilde\om(2)$ has no
$*$ particles and again the conclusion follows immediately.

Using the above-mentioned $\delta$-stability property of the single-cluster couplings
yields a similar property for the ``overall'' coupling of $\calP_{q}$ and $\calP_{q'}$.
The only thing that could go ``wrong'' is if there is a cluster with size $\geq D(\alpha)$. However, by
the claim, this has probability going to $0$ as $n \rightarrow \infty$. The proof of Lemma \ref{stabco} is complete.
\end{pf}

We proceed with the proof of Theorem \ref{mainthm}.
Fix a value $\hat q$ in the interval $(q_1,q_2)$, where $q_1 = 4 \la_1/(4 \la_1 + 1)$ as before [see
below \eqref{lbQcr}] and $q_2 = 4 \la_2/(4 \la_2 + 1)$.

Now we are ready to introduce $0-1$ valued random variables to which we can apply the results in Section \ref{Sharp-threshold results}.
Let the box $B_n$ and the space--time region $\mathit{ST}(n)$ be as before (see a few lines before Definition \ref{stab-def}).
Now partition every
time axis in intervals of length $\de$, with $\de$ as defined just before Definition \ref{stab-def}.

As before, we have on each time axis a Poisson point process with density $1$ and each Poisson point is, independently of the
others, of type $*$ with probability $1-q$ and of each of the types $\rightarrow$, $\leftarrow$, $\uparrow$, $\downarrow$ with
probability $q/4$.
Let $v \in B_n$ and $k \in \N$,
$0 \leq k \leq n/\de$. By the $k$th interval of $v$ for the above-mentioned partition,
we will mean $\{v\} \times (-k \de, (-k+1) \de]$,
and we define
\[
X_{*}^{(v,k,\de)} := I\{\exists \mbox{ a Poisson point of type } * \mbox{ in the } k \mbox{th interval of } v\}.
\]

Similarly define
\[
X_{\rightarrow}^{(v,k,\de)} := I\{\exists \mbox{ a Poisson point of type } \rightarrow \mbox{ in the } k \mbox{th interval of } v\}
\]
and, analogously,
$X_{\leftarrow}^{(v,k,\de)}$, $X_{\uparrow}^{(v,k,\de)}$ and $X_{\downarrow}^{(v,k,\de)}$.
Note that this is a collection of independent $0-1$ valued random variables.

Recall the definition of $\et_v^{(n)}$ and ${\bnu}_{\langle q\rangle}^{(n)}$ below equation \eqref{lbQcr}.
The $X$ variables defined above give only ``crude'' information about the space--time diagram; they
tell which of the types $*$, $\rightarrow$, etc., occur in each interval but they
do not tell their precise locations inside the intervals. Nevertheless, this
incomplete information is often enough to conclude that there is a certain space--time path.
Let $\et_v^{(n,\de)}$ be the indicator of the event that the
values of the $X^{(\cdot,\cdot,\de)}$ variables imply that $\et_v^{(n)} = 1$.

\begin{rem*} Note that if $\et_v^{(n)} = 1$ then,
after for some Poisson points with mark~$*$, this
mark is replaced by an arrow still $\et_v^{(n)} = 1$. The same remark holds for $\et_v^{(n,\de)}$ instead
of $\et_v^{(n)}$.
\end{rem*}

It is easy to see that
\begin{eqnarray}
\label{etxnI-bound}
\et_v^{(n)} \geq \et_v^{(n,\de)} &\geq& I\bigl\{\exists (w,t) \mbox{ with } d(v,w)  = \bigl\lfloor \sqrt n \bigr\rfloor \mbox{ or } t = -\sqrt n
\mbox{ s.t.}\nonumber
\\[-8pt]\\[-8pt]
  &&\hspace*{10pt} \exists   \delta \mbox{-stable space--time path from } (w,t) \mbox{ to }(v,0)\bigr\}.\nonumber
\end{eqnarray}

Hence, with the following notation (where $R$ is a box)
\[
H^{(n,\de)}(R) := \bigl\{ \exists   \et^{(n,\de)} \mbox{-occupied horizontal crossing of } R \bigr\},
\]
we get
\begin{eqnarray*}
\nonumber
 && {\mathcal P}_{\hat q}\bigl(H^{(n_i,\de)}(L_{n_i})\bigr) \geq
{\mathcal P}_{q_1}\bigl(\exists   \et^{(n_i)} \mbox{-occupied horizontal crossing of } L_{n_i}\bigr) - \varepsilon(n_i)
\\
  &&\qquad = {\bnu}_{\langle q_1\rangle}^{(n_i)}\bigl(H(L_{n_i}) - \varepsilon(n_i)\bigr), \nonumber
\end{eqnarray*}
where $\varepsilon(n)$ is a function of $n$ that goes to $0$ as $n \rightarrow \infty$ and where
the inequality comes from the second inequality in \eqref{etxnI-bound} and Lemma \ref{stabco} and the equality comes directly from the definitions.

By \eqref{lbLcr2sig}, and obvious monotonicity [see the Remark preceding \eqref{etxnI-bound}], this gives the following
lemma.

\begin{lem}
\label{HILni-bound}
For each choice of $\al$ the following holds for all sufficiently large~$i$:

%e16 ###
\begin{equation}
\label{HILni-bound-eq}
{\mathcal P}_q\bigl(H^{(n_i, \de)}(L_{n_i})\bigr) > \frac{\tep}{2}, \qquad    q \geq {\hat q}.
\end{equation}
\end{lem}

Now we ``wrap around the box $B_{n_i}$ horizontally'' by identifying every vertex $(6 n_i, y)$ on $B_{n_i}$ with the vertex $(0,y)$ thus turning this box
into a cylinder.
Define $\et_v^{(n_i, \de, C)}$ as the natural analog for the cylinder of $\et_v^{(n_i,\de)}$.

\begin{rem*}``By the truncation to distance $\sqrt n$ of these variables'' and because the left- and right-hand side of $L_{n_i}$ have distance larger
than $\sqrt n$ to the boundary of $B_{n_i}$,
the event that there is a $\et_{\cdot}^{(n_i, \de, C)}$-occupied horizontal crossing of $L_{n_i}$ and the event that
there is a $\et_{\cdot}^{(n_i, \de)}$-occupied horizontal crossing of $L_{n_i}$ are the same.
\end{rem*}

Let $A^{(n_i,\de)}$ be the event that at least one of the $(6 n_i -1)$ horizontal translates
of $L_{n_i}$ on this cylinder has an $\et_{\cdot}^{(n_i, \de, C)}$-occupied horizontal
crossing.
Note that the event $A^{(n_i,\de)}$ is still defined in terms of the random variables
$X^{{\cdot},k, \de}$ defined earlier. Moreover, this event is increasing in the $X$ variables corresponding
with arrows and decreasing in those corresponding with $*$'s. For each choice of $\al$ the
following holds for all sufficiently large~$i$:
%e17 ###
\begin{equation}
\label{Ani-bound-eq}
{\mathcal P}_q\bigl(A^{(n_i, \de)}\bigr) \geq {\mathcal P}_q\bigl(H^{(n_i, \de)}(L_{n_i})\bigr) > \frac{\tep}{2},  \qquad   q \geq {\hat q},
\end{equation}
where the first inequality is (taking into account the above remark) trivial and the last inequality is exactly Lemma \ref{HILni-bound}.

As stated before the $X$ variables are independent $0-1$ valued random variables.
Further, for each $v$ and $k$, $X_*^{(v,k, \de)}$ has probability $1 - \exp(-(1-q) \de)$ to be $1$.
Each random variable $X_{\rightarrow}^{(v,k, \de)}$ has probability $1 - \exp(-\de q/4)$ to be $1$.
The same holds for the other three arrow types.

Also note that the event $A^{(n_i, \de)}$ is partially symmetric in the following sense:
for fixed value $k$ and fixed $0 \leq l \leq 3 n$ all variables $X_{\rightarrow}^{v,k, \de}$ with $v \in B_n$ with $y-$ coordinate $l$,
``play the same role.'' In particular,
each of them has the same probability to be pivotal for the event $A^{(n_i, \de)}$.
The same statement holds for each of the other three arrow types and for type $*$. Further note that
for each $k$ and $l$ the number of such random variables $X_{\rightarrow}^{(v,k, \de)}$ is of order $n$.
Again, the same statement holds for each of the other types.

We will apply Corollary \ref{talm-cor-dif} with $m$ equal to our ``current'' $n$ and with $p_1$ and $p_2$
equal to
$1 - \exp(-\de q/4)$ and $1 - \exp(-(1-q) \de)$, respectively. For our purpose we should think
of $n$ as very large and hence, $\de$ very small. For fixed $n$ (and hence, $\delta$), the $p_1$ and
$p_2$ above are functions of $q$ and
\[
\frac{d p_1}{d q} = \frac{\de}{4} \exp(- \de q/4)
\]
which is of order $\delta$. More precisely, there are positive constants $C'$ and $C''$ such that
\[
C' \de \leq \frac{d p_1}{d q} \leq C'' \de\qquad \mbox{for all } \de \in (0,1) \mbox{ and }
q \in [\hat q, q_2).
\]
Similarly, $p_1$ and $p_2$ are also of order $\de$ and $d p_2 / d q$ is of order $-\de$.
Therefore, when we take the derivative with respect to $q$ of the probability of the event $A^{(n_i, \de)}$,
the factor of order $\de$ that comes from $\max(p_1,p_2)$ in the denominator in the right-hand side of
\eqref{talm-cor-dif-eq} is canceled by a factor of order $\de$ that comes from
$d p_1 / d q$ and $d p_2 / d q$.
Essentially the only ``remaining'' effect of $\de$ comes
from the logarithmic expression in the denominator in the right-hand side of \eqref{talm-cor-dif-eq}.
More precisely what we get is
%
%e18 ###
\begin{equation}
\label{diff-bound}
\quad\ \frac{d}{d q} {\mathcal P}_q\bigl(A^{(n_i, \de)}\bigr) \geq \frac{ C_9    {\mathcal P}_q(A^{(n_i, \de)})
(1 - {\mathcal P}_q(A^{(n_i, \de)}) \log n}{\log(2/\de)},  \qquad   q \in[\hat q, q_2),
\end{equation}
where $C_9 > 0$ depends on $\hat q$ and $q_2$ only.

Let $\ep^* > 0$. By \eqref{Ani-bound-eq}, \eqref{diff-bound} and because ${\mathcal P}_q(A^{(n_i, \de)})$ is
clearly nondecreasing in $q$,
it follows that, for every choice of $\al$, the following holds for all sufficiently large~$i$:
If ${\mathcal P}_{q_2}(A^{(n_i, \de)}) < 1 - \ep^*$ then, for all $q \in [\hat q, q_2)$,
\[
\frac{d}{d q} {\mathcal P}_q\bigl(A^{(n_i, \de)}\bigr) \geq C_9 \frac{\tilde\ep}{2} \ep^*
\frac{\log n}{\log(2/\de)} \geq \frac{C_{10} \tilde\ep \ep^*}{\al}
\]
(where the last inequality used that $\de = n^{-\al}$) and hence,
\[
{\mathcal P}_{q_2}\bigl(A^{(n_i, \de)}\bigr) \geq (q_2 - {\hat q}) C_{10} {\tilde\ep} \ep^*/ \al.
\]

By choosing $\al$ sufficiently small this gives the following lemma.

\begin{lem}
\label{q2Ani-bound}
For every $\ep^* > 0$ there is an $\al >0$ such that for all sufficiently large~$i$
%e19 ###
\begin{equation}
{\mathcal P}_{q_2}\bigl(A^{(n_i, \de)}\bigr) > 1 - \ep^*.
\end{equation}
\end{lem}

Now if there is a horizontal crossing of one of the above-mentioned translates of $L_{n_i}$,
there must be a horizontal crossing in the ``hard'' direction of at least one of the following (six) translates
(on the cylinder) of the rectangle $[0, 3 n_i] \times [n_i, 2 n_i]$:
\[
[j n_i, (j+3)n_i (\operatorname{mod} 6 n_i)] \times [n_i, 2 n_i],   \qquad 0 \leq j \leq 5.
\]

Hence, by the usual ``square root trick,''
\[
{\mathcal P}_{q_2}\bigl(H^{(n_i, \de)}([0, 3 n_i] \times [n_i, 2 n_i])\bigr) \geq
1 - \bigl(1 - {\mathcal P}_{q_2}\bigl(A^{(n_i, \de)}\bigr)\bigr)^{1/6}
\]
which, combined with Lemma \ref{q2Ani-bound}, immediately gives that for every $\ep^* >0$ there is an $\al >0$ s.t. for
all sufficiently large~$i$
%
%e20 ###
\begin{equation}
\label{semi-fin-eq}
{\mathcal P}_{q_2}\bigl(H^{(n_i, \de)}([0, 3 n_i] \times [0, n_i])\bigr) > 1 - \ep^*.
\end{equation}

Finally the following proposition is obtained.

\begin{prop}
\label{H-bound}
%
%e21 ###
\begin{equation}
\label{fin-eq}
\lim_{i \rightarrow \infty} {\bar \nu}_{\la_2}(H(3 n_i, n_i)) = 1.
\end{equation}
\end{prop}

\begin{pf}
Let $\ep^* > 0$ be given.
By \eqref{q-la-trans} (and the definition of $q_2$),
$\bnu_{\la_2}(H(3 n_i,\break n_i)) = \bnu_{\langle q_2\rangle}(H(3 n_i, n_i))$. Hence, by \eqref{expnsig},
$\liminf_{i \rightarrow \infty} \bnu_{\la_2}(H(3 n_i, n_i))$ is equal to $\liminf_{i \rightarrow \infty} {\bnu}_{\langle q_2\rangle}^{(n_i)}(H(3 n_i, n_i))$
which by the first inequality in \eqref{etxnI-bound} is larger than or equal to
\[
\liminf_{i \rightarrow \infty} {\mathcal P}_{q_2}\bigl(H^{(n_i, \de)}([0, 3 n_i] \times [0, n_i])\bigr).
\]
This last expression is, by (the statement ending with) \eqref{semi-fin-eq} and a suitable choice of $\al$,
larger than $1 - \ep^*$.
Summarizing, we have that for every $\ep^* >0$, $\liminf_{i \rightarrow \infty} \bnu_{\la_2}(H(3 n_i, n_i))$ is larger than $1 - \ep^*$.
\end{pf}

Proposition \ref{H-bound}, together with the finite-size criterion Lemma \ref{lem-fs}, immediately yields
${\bnu}_{\la_2}(|\calC_O| = \infty) > 0$
which, as observed in the beginning of this section, completes the proof of Theorem \ref{mainthm}. \\

\section*{Acknowledgments} I thank Oliver Riordan for a very useful discussion concerning influence and sharp-threshold results and in particular for
pointing out Corollary \ref{talm-cor} for general $m$. The work in this paper was partly done during visits to
the Isaac Newton Institute in June 2008, the Institut Henri Poicar\'e in October 2008 and the Institut Mittag Leffler in April and June 2009.
I thank these institutes for their support and hospitality. Finally, I thank Jeff Steif and Geoffrey Grimmett for
several useful discussions and an anonymous referee for many valuable detailed comments and suggestions for improvement
on an earlier version of this paper.

% imsref loaded by elazauskaite, 2010-08-12 11:34:19

\printaddresses

\end{document}